\documentclass[a4paper,12pt]{amsart}
\usepackage{standalone}
\usepackage{amsmath}
\usepackage{amssymb}
\usepackage{mathrsfs}
\usepackage{ifthen}
\usepackage{graphicx}
\usepackage[T1]{fontenc} 
\usepackage{tikz}
\usetikzlibrary{arrows}
\usetikzlibrary{shapes}
\usetikzlibrary{patterns}
\usepackage{xcolor}

\nonstopmode \numberwithin{equation}{section}
\newtheorem{thm}{Theorem}[section]
\newtheorem{cor}[equation]{Corollary}
\newtheorem{lem}[equation]{Lemma}

\theoremstyle{definition}
\newtheorem{defn}{Definition}[section]

\newtheorem{prob}[equation]{Problem}
\newtheorem{rem}{Remark}[section]
\newtheorem{observation}{Observation}

\newcounter{minutes}
\divide\time by 60
\newcounter{hours}
\multiply\time by 60
\addtocounter{minutes}{-\time}

\newcounter {own}
\def\theown {\thesection       .\arabic{own}}

\newenvironment{pf}[1][]{%
	\vskip 3mm
	\noindent
	\ifthenelse{\equal{#1}{}}%
	{{\slshape Proof. }}%
	{{\slshape #1.} }%
}%
{\qed\bigskip}

\newcounter{alphabet}

\newenvironment{Thm}[1][]{\refstepcounter{alphabet}%
	\bigskip%
	\noindent%
	{\bf Theorem \Alph{alphabet}}%
	\ifthenelse{\equal{#1}{}}{}{ (#1)}%
	{\bf .} \itshape}{\vskip 8pt}

\def\be{\begin{equation}}
	\def\ee{\end{equation}}

\newcommand{\bee}{\begin{enumerate}}
	\newcommand{\eee}{\end{enumerate}}

\newcommand{\blem}{\begin{lem}}
	\newcommand{\elem}{\end{lem}}
\newcommand{\bthm}{\begin{thm}}
	\newcommand{\ethm}{\end{thm}}
\newcommand{\bcor}{\begin{cor}}
	\newcommand{\ecor}{\end{cor}}
\newcommand{\beg}{\begin{examp}}
	\newcommand{\eeg}{\end{examp}}
\newcommand{\begs}{\begin{examples}}
	\newcommand{\eegs}{\end{examples}}
\newcommand{\bdefe}{\begin{defin}}
	\newcommand{\edefe}{\end{defin}}
\newcommand{\bprob}{\begin{prob}}
	\newcommand{\eprob}{\end{prob}}
\newcommand{\bei}{\begin{itemize}}
	\newcommand{\eei}{\end{itemize}}

\begin{document}
	
	\title{A note on the equivalence of Gromov boundary and metric boundary}
	
	\author{Vasudevarao Allu}
	\address{Vasudevarao Allu,
		Department of Mathematics,
		School of Basic Sciences,
		Indian Institute of Technology  Bhubaneswar,
		Argul, Bhubaneswar, PIN-752050, Odisha (State),  India.}
	\email{avrao@iitbbs.ac.in}

	\author{Abhishek Pandey}
	\address{Abhishek Pandey,
		Department of Mathematics,
		School of Basic Sciences,
		Indian Institute of Technology  Bhubaneswar,
		Argul, Bhubaneswar, PIN-752050, Odisha (State),  India.}
	\email{ap57@iitbbs.ac.in}
	
	\makeatletter
	\@namedef{subjclassname@2020}{\textup{2020} Mathematics Subject Classification}
	\makeatother

	\subjclass[2020]{Primary 30L10, 30L99, 30C65, 51F30, 53C23}
	\keywords{Gromov hyperbolic spaces, Gromov boundary, metric boundary, quasihyperbolically visible spaces, quasihyperbolic metric, quasihyperbolic geodesics, visibility}

	\def\thefootnote{}
	\footnotetext{ {\tiny File:~\jobname.tex,
			printed: \number\year-\number\month-\number\day,
			\thehours.\ifnum\theminutes<10{0}\fi\theminutes }
	} \makeatletter\def\thefootnote{\@arabic\c@footnote}\makeatother

	\thanks{}

	\begin{abstract}
		In this paper, we introduce the concept of quasihyperbolically visible spaces. As a tool, we study the connection between the Gromov boundary and the metric boundary.
	\end{abstract}
	
	\maketitle
	\pagestyle{myheadings}
	\markboth{Vasudevarao Allu and Abhishek Pandey}{A note on the equivalence of Gromov boundary and metric boundary}

\section{Introduction}
For a metric space $(\Omega,d)$, its metric completion and metric boundary are denoted by $\overline{\Omega}^d$ and $\partial_d\Omega=\overline{\Omega}^d\setminus \Omega$, respectively. Let $(\Omega,d)$ be a rectifiably connected, locally compact, non-complete metric space. We say that $(\Omega,d)$ is \emph{minimally nice} if the identity map from $(\Omega,d)$ to $(\Omega,\lambda_{\Omega})$ is continuous, where $\lambda_{\Omega}$ is the length metric on $\Omega$ with respect to $d$ (see Definition \ref{metric definition}). Note that every proper domain in $\mathbb{R}^n$ is minimally nice. Consider the metric space $(\Omega,k_{\Omega})$, where $k_{\Omega}$ is the quasihyperbolic metric on the metric space $(\Omega,d)$ (see Section \ref{QH metric}). If  $(\Omega, k_{\Omega})$ is Gromov hyperbolic, then its  Gromov boundary and Gromov compactification are denoted by $\partial_{G}\Omega$ and  $\overline{\Omega}^G$, respectively (see Section \ref{GHS} for the definitions).
Bonk, Heinonen, and Koskela, in their seminal work \cite{BHK} on the uniformization of Gromov hyperbolic spaces, established a two-way correspondence between uniform metric spaces and Gromov hyperbolic spaces (see \cite[Theorem 1.1]{BHK}). The following important result in \cite{BHK} says that uniform domains in $\mathbb{R}^n$ can be characterized in terms of Gromov hyperbolicity.

\begin{Thm}\cite[Theorem 1.11]{BHK}
	Let $\Omega\subset \mathbb{R}^n$ be a bounded domain. Then the following are equivalent:
	\begin{itemize}
		\item[(i)] $\Omega$ is uniform.
		\item[(ii)] $(\Omega,k_{\Omega})$ is Gromov hyperbolic and the identity map $\operatorname{id}_\Omega:(\Omega,k_{\Omega})\to(\Omega,d_{Euc})$ extends as a homeomorphism $\Phi:\overline{\Omega}^G\to \overline{\Omega}^{Euc}$ such that $\Phi$ is a quasisymmetric homeomorphism from $\partial_{G}\Omega$ onto $\partial_{Euc}\Omega$.
	\end{itemize}
\end{Thm}

Similar results were obtained in the settings of Banach spaces and metric measure spaces by V\"ais\"al\"a \cite[Theorem 2.29]{Vaisala-GH-1} and by Herron, Shanmugalingam, and Xie \cite{HSX-2008}, respectively. A central issue in this context is to determine when, and in what sense, the identity map 
$\operatorname{id}_\Omega:(\Omega,k_\Omega)\to (\Omega,d)$ 
extends to the boundary. Much of the existing literature (see \cite{BHK, Vaisala-GH-1, HSX-2008}) seeks additional regularity of such an extension, for instance quasisymmetric or quasim\"obius behavior. In contrast, in this paper we focus only on the \emph{existence} of a boundary extension with basic properties, namely continuity and surjectivity (and, under additional assumptions, that it is a homeomorphism). This leads to the following general question.
\begin{prob}\label{GB-MB}
	Let $(\Omega,d)$ be a minimally nice space. Suppose $(\Omega,k_{\Omega})$ is Gromov hyperbolic. When is the following true?
	
	\begin{quote}
		The identity map $\operatorname{id}_{\Omega}: (\Omega,k_{\Omega})\to (\Omega,d)$ extends to a continuous surjective map, or to a homeomorphism, $\widehat{\operatorname{id}_{\Omega}}:\partial_{G}\Omega\to \partial_d\Omega$.
	\end{quote}
\end{prob}

Problem \ref{GB-MB} is an interesting problem in the field of metric space analysis. The work of Lammi \cite{Lammi-2011,Lammi-Thesis} is the most general work with regard to Problem \ref{GB-MB}. Note that to have a continuous surjective, or a homeomorphic, extension, we need $\overline{\Omega}^d$ to be compact; hence $\partial_d\Omega$ must be compact, since the Gromov boundary $\partial_{G}\Omega$ is compact. But in general, the metric boundary $\partial_{d}\Omega$ need not be compact. For instance, let
$$
\Omega=(0,1)\times (0,1)\setminus \bigcup_{j=1}^\infty \left(\left\{\frac{1}{2^j}\right\}\times \left[0,\frac{1}{2}\right]\right)
$$
and equip $\Omega$ with the inner metric $\lambda_{\Omega}$ (see Definition \ref{metric definition}). Then the inner boundary is closed and bounded but not compact. Indeed, let $\{x_j\}\subset \partial_{I}\Omega$ be the sequence of midpoints
$$
x_j=\left(\frac{3}{2^{j+2}},0\right).
$$
It is easy to see that $\lambda_{\Omega}(x_{j},x_{j+1})\ge 1$ for every $j\in \mathbb{N}$, and hence this sequence does not have a convergent subsequence. The following result of Lammi \cite{Lammi-Thesis} says that, under the quasiconvexity assumption, compactness of the metric boundary is necessary and sufficient for the homeomorphic extension.

\begin{thm}\cite[Theorem 1.1]{Lammi-Thesis}\label{topological condition}
	Let $(\Omega,d)$ be a locally compact and non-complete quasiconvex space. Assume that $(\Omega, k_{\Omega})$ is Gromov hyperbolic and the Gehring--Hayman theorem holds in $(\Omega,d)$. Then the following are equivalent:
	\begin{itemize}
		\item[(i)] $\partial_{d}\Omega$ is compact.
		\item[(ii)] $\overline{\Omega}^d=\Omega\cup\partial_{d}\Omega$ is compact.
		\item[(iii)] The identity map $\operatorname{id}_{\Omega}: (\Omega,k_{\Omega})\to (\Omega,d)$ extends as a homeomorphism from $\partial_{G}\Omega$ onto $\partial_d\Omega$.
	\end{itemize}
\end{thm}

Further, the following result of Lammi \cite{Lammi-2011} gives an analytic condition for the homeomorphic extension.

\begin{thm}\cite[Theorem 1.1]{Lammi-2011}\label{analytic condition}
	Let $(\Omega,d)$ be a locally compact and non-complete quasiconvex space. Assume that $(\Omega, k_{\Omega})$ is Gromov hyperbolic and the Gehring--Hayman theorem holds in $(\Omega,d)$.
	Let $x_0\in \Omega$ be a base point and suppose that the quasihyperbolic metric satisfies
	$$
	k_{\Omega}(x_0,x)\le \phi\left(\frac{\delta_{\Omega}(x_0)}{\delta_{\Omega}(x)}\right),
	$$
	where $\phi:(0,\infty)\to (0,\infty)$ is an increasing function such that
	$$
	\sum_{j=1}^{\infty}\frac{1}{\phi^{-1}(j)}<\infty.
	$$
	Then the identification map $\partial_{G}\Omega\to \partial_d\Omega$ is a homeomorphism; moreover, $\partial_d\Omega$ is compact.
\end{thm}

Throughout, for $z\in\Omega$ we write
	\[
	\delta_\Omega(z):=d(z,\partial_d\Omega)=\inf\{d(z,w):w\in\partial_d\Omega\}.
	\]
It is important to mention that Lammi's proof of Theorem \ref{analytic condition} does not use a compactness argument. Indeed, the compactness of $\partial_d\Omega$ comes as a consequence of the homeomorphic extension.

\begin{rem}[\textbf{Equivalence of Gromov boundary and inner boundary}]
	Let $\Omega$ be a bounded domain in $\mathbb{R}^n$. Equip it with the inner metric $\lambda_{\Omega}$, and let $\partial_I\Omega$ be the inner boundary, i.e.\ the metric boundary of the metric space $(\Omega, \lambda_{\Omega})$. Assume that $(\Omega,k_{\Omega})$ is Gromov hyperbolic. Consider the following problem.
	\begin{prob}\label{GB-IB}
		For which bounded domains $\Omega\subset\mathbb{R}^n$ does the identity map $\operatorname{id}_\Omega:(\Omega,k_{\Omega})\to(\Omega, \lambda_\Omega)$ extend to a homeomorphism from $\partial_{G}\Omega$ onto $\partial_I\Omega$?
	\end{prob}
	Lammi's work \cite{Lammi-2011,Lammi-Thesis} is very helpful for Problem \ref{GB-IB}. This is because, as a length space, $(\Omega,\lambda_{\Omega})$ is a quasiconvex space, and hence Theorem \ref{topological condition} tells us that the compactness of the inner boundary is sufficient for the homeomorphic extension. Moreover, Theorem \ref{analytic condition} gives us an analytic condition for the inner boundary to be compact.
\end{rem}

The growth condition in Theorem \ref{analytic condition} implies that $(\Omega,d)$ is bounded. Recently, Zhou \textit{et al.} \cite[Theorem 1.8]{Zhou-2022} studied this problem for unbounded Gromov hyperbolic spaces. Further, Zhou and Rasila \cite{Rasila-2021} studied Problem \ref{GB-MB} in the case where $\Omega$ is a generalized John domain in $\mathbb{R}^n$ and $d$ is the inner metric.

The main aim of this work is to study Problem \ref{GB-MB}. This study extends our previous work \cite{Allu-2023}, in which we examined the Euclidean variant of Problem \ref{GB-MB}, namely $\Omega$ as a bounded domain in $\mathbb{R}^n$ and $d$ as the Euclidean metric. In that setting, we introduced the concept of QH-visibility domains and provided a complete solution (see \cite[Theorems 2.2 and 2.3]{Allu-2023}). Motivated by this, we study Problem \ref{GB-MB} via the visibility property and introduce the concept of quasihyperbolically visible spaces (see Definition \ref{qh visible} below). Our approach is motivated by the recent work of Bharali and Zimmer \cite{Bharali-2017} and of Bracci, Nikolov, and Thomas \cite{Bracci-2022}, where they studied a suitable form of a notion called visibility. The concept of visibility goes back to Eberlein and O'Neill \cite{O'Neill-1973}, who introduced a general construction of compactification of Hadamard manifolds by attaching the boundary at infinity with the cone topology. Using the notion of boundary at infinity, Eberlein and O'Neill defined the term visibility manifold for a complete Riemannian manifold with nonpositive sectional curvature. Bharali and Zimmer \cite{Bharali-2017} introduced Goldilocks domains in $\mathbb{C}^n$, which satisfy a visibility condition for the Kobayashi distance. This led Bharali and Maitra \cite{Bharali-2021} to the notion of visibility domains in the Kobayashi setting. Moreover, our use of quasihyperbolic visibility to control the cluster set of a quasihyperbolic geodesic ray at infinity is inspired by the method of Bharali--Maitra \cite{Bharali-2021}, where cluster sets of dynamical orbits are controlled via visibility property.

\begin{defn}\label{qh visible}
	Let $(\Omega,d)$ be a minimally nice space. We say that
	\begin{itemize}
		\item[(1)] $(\Omega,k_{\Omega})$ has the {\it visibility property} if the following holds:
		\begin{itemize}
			\item[($*$)] For every pair of distinct points $p, q\in \partial_d\Omega$, there exists a compact set $K\subset\Omega$ such that for any sequences $(x_n), (y_n)$ in $\Omega$ with $x_n\stackrel{d}{\to} p$ and $y_n\stackrel{d}{\to} q$, every quasihyperbolic geodesic $\gamma_{n}$ joining $x_n$ and $y_n$ intersects $K$.
		\end{itemize}
		\item[(2)] $(\Omega,d)$ satisfies the {\it quasihyperbolic visibility property} if $(\Omega,k_{\Omega})$ has the visibility property.
		\item[(3)] $(\Omega,d)$ is {\it quasihyperbolically visible} if $(\Omega,d)$ has the quasihyperbolic visibility property.
	\end{itemize}
\end{defn}

\section{Results}
We now state our main results. Throughout, we assume that $\overline{\Omega}^d$ is compact.

\begin{thm}\label{Main-1}
	Let $(\Omega,d)$ be a bounded minimally nice space. Assume that $\overline{\Omega}^d$ is compact and $(\Omega,k_{\Omega})$ is Gromov hyperbolic. Then the following are equivalent:
	\begin{itemize}
		\item[(i)] $(\Omega,d)$ is quasihyperbolically visible,
		\item[(ii)] The identity map $\operatorname{id}_{\Omega}$ extends to a continuous surjective map $\Phi:\partial_{G}\Omega\to \partial_d\Omega$.
	\end{itemize}
	In addition, if $(\Omega,d)$ is quasiconvex and the Gehring--Hayman theorem holds in $(\Omega,d)$, then the extension map $\Phi:\partial_{G}\Omega\to \partial_d\Omega$ is a homeomorphism.
\end{thm}

\begin{rem}
	By saying that the Gehring--Hayman theorem holds in $(\Omega,d)$, we mean that there exists a constant $C_{gh} \geq 1$ such that for each pair of points $x,y \in \Omega$ and for each quasihyperbolic geodesic $\gamma$ joining $x$ and $y$, it holds that
	$$\ell_d(\gamma[x,y]) \leq C_{gh}\,\ell_d(\sigma),$$
	where $\sigma$ is any other curve joining $x$ and $y$ in $\Omega$. In other words, quasihyperbolic geodesics are essentially the shortest curves in $\Omega$. We refer to \cite{Koskela-2013,Koskela-2014} for conditions ensuring that $(\Omega,d)$ satisfies the Gehring--Hayman theorem.
\end{rem}

Our next result gives criteria for a minimally nice space to be a quasihyperbolically visible space.

\begin{thm}\label{Main-2}
	Let $(\Omega,d)$ be a bounded minimally nice space. Fix a base point $x_0\in \Omega$ and let $\phi:(0,\infty)\to(0,\infty)$ be a strictly increasing function with $\phi(t)\to \infty$ as $t\to \infty$ such that
	\begin{equation}\label{Int}
		\int_{0}^{\infty}\frac{dt}{\phi^{-1}(t)}<\infty.
	\end{equation}
	If the quasihyperbolic metric satisfies
	\begin{equation}\label{Growth}
		k_{\Omega}(x_0,x)\le \phi\left(\frac{\delta_{\Omega}(x_0)}{\delta_{\Omega}(x)}\right),
	\end{equation}
	then $(\Omega,d)$ is quasihyperbolically visible.
\end{thm}

\begin{rem}
	It is easy to see that Theorem \ref{Main-1} together with Theorem \ref{Main-2} gives Theorem \ref{analytic condition}. Furthermore, Theorem \ref{Main-1} and Theorem \ref{Main-2} together imply a special case of \cite[Theorem 1.8(1)]{Zhou-2022}, which we explain below.
	\begin{pf}[\textbf{Explanation}]
		Let $(\Omega,d)$ be a rectifiably connected, locally compact, non-complete bounded metric space. It is known that if $(\Omega,d)$ is bounded and $\varphi$-uniform, then there is an increasing function $\phi:[0,\infty)\to [0,\infty)$ such that for all $x\in \Omega$
		$$k_{\Omega}(x_0,x)\le \phi\left(\frac{\delta_{\Omega}(x_0)}{\delta_{\Omega}(x)}\right),$$
		where $x_0\in \Omega$ satisfies $\delta_{\Omega}(x_0)=\max_{x\in \Omega}\delta_{\Omega}(x)$. In particular, we can take
		$$\phi(t)=\varphi\left(\frac{\mbox{diam}(\Omega)\,t}{\delta_{\Omega}(x_0)}\right).$$
		Note that the conditions
		\begin{equation}\label{phi-uniform}
			\int_{0}^{\infty}\frac{dt}{\phi^{-1}(t)}<\infty \,\,\, \mbox{ and }\,\,\, \int_{0}^{\infty}\frac{dt}{\varphi^{-1}(t)}<\infty
		\end{equation}
		are mutually equivalent. Hence, by Theorem \ref{Main-2}, a bounded $\varphi$-uniform domain satisfying \eqref{phi-uniform} is quasihyperbolically visible. This, together with Theorem \ref{Main-1} and \cite[Theorem 5.1]{Koskela-2014}, gives Theorem 1.8(1) of \cite{Zhou-2022}.
	\end{pf}
\end{rem}
	
\section{Preliminaries}

\subsection{Metric geometry}
Let $(\Omega,d)$ be a metric space. A curve is a continuous function $\gamma:[a,b]\rightarrow \Omega$. Let $\mathcal{P}$ denote the set of all partitions $a = t_{0}<t_{1}<t_{2}< \cdots<t_{n}=b$ of the interval $[a,b]$. The length of the curve $\gamma$ in the metric space $(\Omega,d)$ is
$$\ell_{d}(\gamma) = \sup_{\mathcal{P}} \sum_{k=0}^{n-1}d(\gamma(t_{k}), \gamma(t_{k+1})).$$
A curve is said to be rectifiable if $\ell_d(\gamma) < \infty$. A metric space $(\Omega,d)$ is said to be rectifiably connected if every pair of points $x,y \in \Omega$ can be joined by a rectifiable curve. For a rectifiable curve $\gamma$, we define the arc-length function $s:[a,b]\rightarrow \left[ 0,\ell_d(\gamma) \right]$ by $s(t) = \ell_{d}(\gamma|_{[a,t]})$. The arc-length function is of bounded variation. For any rectifiable curve $\gamma:[a,b]\to \Omega$, there is a unique map $\gamma_s:[0,\ell_d(\gamma)]\to \Omega$ such that $\gamma=\gamma_s\circ s$, and the curve $\gamma_s$ is called the arclength parametrization of $\gamma$.

\begin{defn}
	A curve $\gamma:[a,b]\to (\Omega,d)$ is said to be a geodesic if it is an {\it isometry}, {\it i.e.}\ for all $t,t'\in [a,b]$,
	$d(\gamma(t), \gamma(t'))=|t-t'|$.
	An isometry $\gamma:[0,\infty)\to (\Omega,d)$ is called a {\it geodesic ray}, and an isometry $\gamma:(-\infty,\infty)\to (\Omega,d)$ is called a {\it geodesic line}.
\end{defn}

\begin{defn}\label{metric definition}
	A metric space $(\Omega,d)$ is said to be
	\begin{itemize}
		\item[(1)] {\it a geodesic space} if every pair of points $x,y \in \Omega$ can be joined by a geodesic;
		\item[(2)] {\it proper} if every closed ball is compact in $\Omega$;
		\item[(3)] {\it intrinsic} (or a {\it path metric space}, or a {\it length space}) if for all $x,y\in \Omega$ it holds that
		$$d(x,y)=\lambda_\Omega(x,y):=\inf_{\gamma} \ell_d(\gamma),$$
		where the infimum is taken over all rectifiable curves $\gamma$ in $\Omega$ joining $x$ and $y$. The metric $\lambda_\Omega$ is known as the inner length metric associated to $(\Omega,d)$;
		\item[(4)] {\it quasiconvex} if there exists a constant $A\ge 1$ such that every pair of points $x,y\in \Omega$ can be joined by a curve $\gamma$ with $\ell_d(\gamma)\le A\,d(x,y)$.
	\end{itemize}
\end{defn}

\subsection{Gromov hyperbolic space}\label{GHS}
Let $(\Omega,d)$ be a geodesic metric space. A geodesic triangle is the union of geodesics $\gamma_i:[a_i,b_i]\to \Omega$, $i=1,2,3$, such that $a_i<b_i$ for each $i$ and
$\gamma_1(b_1)=\gamma_2(a_2)$, $\gamma_2(b_2)=\gamma_3(a_3)$, and $\gamma_3(b_3)=\gamma_1(a_1)$.
The geodesics $\gamma_1,\gamma_2,$ and $\gamma_3$ are called the sides of the geodesic triangle.

\begin{defn}
	Let $\delta\ge 0$. A geodesic metric space $(\Omega,d)$ is said to be $\delta$-hyperbolic if every geodesic triangle is $\delta$-thin, i.e.\ each side of the geodesic triangle is contained in the $\delta$-neighborhood of the union of the other two sides. A geodesic metric space $(\Omega,d)$ is said to be Gromov hyperbolic if it is $\delta$-hyperbolic for some $\delta\ge 0$.
\end{defn}

Let $(\Omega,d)$ be a proper geodesic Gromov hyperbolic space and let $p\in \Omega$. Let $\mathcal{G}_{p}$ be the space of all geodesic rays $\gamma:[0,\infty)\to \Omega$ with $\gamma(0)=p$, endowed with the topology of uniform convergence on compact subsets of $[0,\infty)$. Define an equivalence relation $\sim$ on $\mathcal{G}_{p}$ as follows: for $\gamma,\sigma\in \mathcal{G}_{p}$,
$$\gamma\sim \sigma \iff \sup_{t\ge 0} d(\gamma(t),\sigma(t))< \infty.$$

\begin{defn}
	\begin{itemize}
		\item[(1)] The {\it Gromov boundary} of $\Omega$ is denoted by $\partial_{G} \Omega$ and is defined as the quotient space $\mathcal{G}_{p}/\sim$ endowed with the quotient topology.
		\item[(2)] The {\it Gromov closure} of $\Omega$ is denoted by $\overline{\Omega}^{G}$ and is defined as $\overline{\Omega}^{G}=\Omega\cup \partial_{G}\Omega$.
	\end{itemize}
\end{defn}

The set $\overline{\Omega}^{G}$ admits a natural topology, and with this topology $\overline{\Omega}^{G}$ is a compactification of $\Omega$, known as the Gromov compactification of $\Omega$ (see, for instance, \cite[Chapter III.H.3, Proposition 3.7]{Bridson-book}). To describe this topology, we need some additional notation. A generalized ray is either a geodesic segment or a geodesic ray.
\begin{itemize}
	\item For a geodesic ray $\gamma \in \mathcal{G}_p$, define $\operatorname{End}(\gamma):=[\gamma]$.
	\item For a geodesic segment $\gamma:[0,R]\to \Omega$ with $\gamma(0)=p$ (where $R>0$), define $\operatorname{End}(\gamma):=\gamma(R)$.
\end{itemize}

\begin{defn}[Gromov convergence]\label{Gromov topology}
	Let $(\Omega,d)$ be a proper geodesic Gromov hyperbolic space. Fix a base point $p\in \Omega$. A sequence $\{x_n\}_{n\in \mathbb{N}}$ in $\overline{\Omega}^G$ converges to $x\in \overline{\Omega}^G$ (denoted by $x_n \stackrel{Gromov}\longrightarrow x$) if and only if there exist generalized rays $\gamma_n$ with $\gamma_n(0)=p$ and $\operatorname{End}(\gamma_n)=x_n$ such that every subsequence of $(\gamma_n)$ contains a further subsequence that converges uniformly on compact subsets to a generalized ray $\gamma$ with $\gamma(0)=p$ and $\operatorname{End}(\gamma)=x$.
\end{defn}

In view of the Gromov topology of $\overline{\Omega}^G$, we have the following observations for geodesic rays and geodesic lines.
\begin{itemize}
	\item For any geodesic ray $\gamma:[0,\infty)\to \Omega$,
	$\lim_{t\to\infty}\gamma(t)=\operatorname{End}(\gamma)=[\gamma]\in \partial_{G}\Omega$.
	For any geodesic line $\gamma:(-\infty,\infty)\to \Omega$, both limits exist and
	$\lim_{t\to-\infty}\gamma(t)\ne\lim_{t\to \infty}\gamma(t)$.
	\item For a geodesic line $\gamma$, let us denote $\gamma(-\infty)=\lim_{t\to-\infty}\gamma(t)$ and $\gamma(\infty)=\lim_{t\to \infty}\gamma(t)$. Then for each $p\in \Omega$ and $\xi\in \partial_{G}\Omega$ there exists a geodesic ray $\gamma:[0,\infty)\to \Omega$ such that $\operatorname{End}(\gamma)=\xi$,
	\item and for each pair of distinct points $\xi, \eta\in \partial_{G}\Omega$ there exists a geodesic line $\gamma$ with $\gamma(-\infty)=\xi$ and $\gamma(\infty)=\eta$ (see \cite[Chapter III.H.3, Lemmas 3.1 and 3.2]{Bridson-book}). This property is called the \emph{visibility property of Gromov hyperbolic spaces}.
\end{itemize}

\subsection{The Quasihyperbolic metric}\label{QH metric}
Gehring and Palka \cite{Gehring-1976} introduced the quasihyperbolic metric for domains $\Omega\subsetneq \mathbb{R}^n$ as a tool to study quasiconformal homogeneity. The same definition naturally generalizes in the minimally nice metric space setting. 
\begin{defn}
Let $(\Omega,d)$ be a rectifiably connected, locally compact, non-complete metric space. Then the metric boundary $\partial_d\Omega$ is nonempty, and we define the {\it quasihyperbolic metric} $k_{\Omega}$ in $\Omega$ by the length element $|dz|/\delta_{\Omega}(z)$, where $\delta_\Omega(z)=d(z,\partial_d\Omega)$ and $|dz|$ denotes the length element for the metric $d$. Thus,
$k_{\Omega}(x,y)=\inf_{\gamma}\ell_{k}(\gamma)$, where the infimum is taken over all rectifiable curves $\gamma$ in $\Omega$ joining $x$ and $y$, and
$$\ell_{k}(\gamma)=\int_{\gamma}\frac{|dz|}{\delta_{\Omega}(z)}$$
is called the {\it quasihyperbolic length} of $\gamma$.
\end{defn}
In view of \cite[Proposition 2.8]{BHK}, it is well known that if $(\Omega,d)$ is minimally nice, then $(\Omega,k_{\Omega})$ is a complete, proper, and geodesic metric space.

We recall some basic estimates for the quasihyperbolic metric which were first established for domains in $\mathbb{R}^n$ (see \cite[Lemma 2.1]{Gehring-1976}) and also hold in minimally nice spaces (see \cite[Chapter 2]{BHK}). For all $x,y\in \Omega$, we have
\begin{equation}\label{qh-eq-1}
	k_{\Omega}(x,y)\ge \log\left(1+\frac{\lambda_{\Omega}(x,y)}{\min\{\delta_{\Omega}(x),\delta_{\Omega}(y)\}}\right)\ge \log\left(1+\frac{d(x,y)}{\min\{\delta_{\Omega}(x),\delta_{\Omega}(y)\}}\right)\ge \left|\log \frac{\delta_{\Omega}(y)}{\delta_{\Omega}(x)}\right|,
\end{equation}
and if $\gamma$ is a rectifiable curve in $\Omega$ joining $x$ and $y$, then
\begin{equation}\label{qh-eq-2}
	\ell_k(\gamma)\ge \log\left(1+\frac{\ell_d(\gamma)}{\min\{\delta_{\Omega}(x),\delta_{\Omega}(y)\}}\right).
\end{equation}
In view of \eqref{qh-eq-1}, we have the following observation.
\begin{observation}\label{observation}
	If sequences $\{z_n\}$ and $\{w_n\} \subset \Omega$ converge to two different points in $\partial_d\Omega$, then $\lim_{n\to \infty}k_{\Omega}(z_n,w_n)=\infty$.
\end{observation}

At this stage, we mention the following important characterization of Gromov hyperbolicity of the quasihyperbolic metric, due to Bonk, Heinonen, and Koskela \cite[Propositions 7.12 and 7.14]{BHK} and Balogh and Buckley \cite[Theorem 0.1]{Balogh-2002}.

\begin{Thm}\cite[Propositions 7.12 and 7.14]{BHK}; \cite[Theorem 0.1]{Balogh-2002}
	Let $\Omega \subset \mathbb{R}^n$ be a proper subdomain. Then $(\Omega,k_{\Omega})$ is Gromov hyperbolic if and only if $\Omega$ satisfies both a Gehring--Hayman condition and a ball separation condition with respect to the inner metric.
\end{Thm}

\section{Proof of Theorem \ref{Main-1}}
Throughout this section, $(\Omega,d)$ is a bounded minimally nice space and we assume that $\overline{\Omega}^d$ is compact.
Throughout, we parametrize quasihyperbolic geodesics (and geodesic rays) by
$k_\Omega$-arclength; in particular, for a quasihyperbolic geodesic $\gamma$ we have
$k_\Omega(\gamma(s),\gamma(t)) = |s-t|$. We divide the proof of Theorem \ref{Main-1} into several lemmas.

\begin{lem}\label{AV-P7-lem1}
	Let $(\Omega,d)$ be a bounded minimally nice space and let $\gamma$ be a quasihyperbolic geodesic ray in $\Omega$. If for any sequence $t_n\to \infty$ the sequence $\{\gamma(t_n)\}$ $d$-converges to a point $x\in \overline{\Omega}^d$, then $x\in \partial_{d}\Omega$.
\end{lem}

\begin{proof}
	Assume for a contradiction that $x\notin \partial_{d}\Omega$. Then $x\in \Omega$.
	Since $\gamma$ is a quasihyperbolic geodesic ray parametrized by $k_\Omega$-arclength,
	\[
	k_{\Omega}(\gamma(0),\gamma(t_n))=t_n\to\infty \quad \text{as } n\to\infty.
	\]
	On the other hand, $\gamma(t_n)\to x$ in the metric $d$ with $x\in\Omega$, so
	$\delta_\Omega(\gamma(t_n))\to \delta_\Omega(x)>0$. Hence, for $n$ large we have
	$\delta_\Omega(\gamma(t_n))\ge \frac{1}{2}\delta_\Omega(x)$, and therefore
	\[
	k_\Omega(\gamma(t_n),x)\le \frac{2}{\delta_\Omega(x)}\, d(\gamma(t_n),x)\xrightarrow[n\to\infty]{}0.
	\]
	By the triangle inequality,
	\[
	\bigl|k_\Omega(\gamma(0),\gamma(t_n)) - k_\Omega(\gamma(0),x)\bigr|\le k_\Omega(\gamma(t_n),x)\to 0,
	\]
	and consequently $k_\Omega(\gamma(0),\gamma(t_n))\to k_\Omega(\gamma(0),x)<\infty$, a contradiction. Therefore $x\in \partial_{d}\Omega$.
\end{proof}

In the next lemma, we show that in a quasihyperbolically visible space $(\Omega,d)$, for any sequence $t_n\to \infty$ the limit $\lim_{n\to \infty}\gamma(t_n)$ exists in the metric $d$, and hence by Lemma \ref{AV-P7-lem1} it belongs to $\partial_{d}\Omega$.

\begin{lem}\label{AV-P7-lem1.1}
	Let $(\Omega,d)$ be a bounded minimally nice space. Suppose $(\Omega,d)$ is quasihyperbolically visible. If $\gamma$ is a quasihyperbolic geodesic ray in $\Omega$, then for any sequence $t_n\to \infty$ the limit $\lim_{n\to \infty}\gamma(t_n)$ exists in the metric $d$.
\end{lem}

\begin{proof}
Since $\overline{\Omega}^d$ is compact, the sequence $\{\gamma(t_n)\}$ has a convergent subsequence. By Lemma \ref{AV-P7-lem1}, every subsequential limit lies in $\partial_d\Omega$.
	Thus, it suffices to show that $\{\gamma(t_n)\}$ has a unique limit point. Assume for a contradiction that $\{\gamma(t_n)\}$ has two distinct limit points $p,q\in \partial_{d}\Omega$, i.e.\ there are two subsequences $\{\gamma(t_{n_{k}})\}_{k}$ and $\{\gamma(t_{s_{k}})\}_{k}$ such that
	\[
	\lim_{k\to \infty}\gamma(t_{n_{k}})=p\neq q= \lim_{k\to \infty}\gamma(t_{s_{k}}).
	\]
	Passing to a subsequence if necessary, we may assume that $t_{n_{k}}<t_{s_{k}}<t_{n_{k+1}}$ for all $k$.
	Consider the quasihyperbolic geodesic segments $\gamma_k=\gamma|_{[t_{n_{k}},t_{s_{k}}]}$. Since $(\Omega,d)$ is a quasihyperbolically visible space, there exists a compact set $K\subset \Omega$ such that, for all sufficiently large $k$, the segment $\gamma_k$ meets $K$.
	Thus, for each such $k$ we can pick $t'_k\in (t_{n_k},t_{s_k})$ with $\gamma(t'_k)\in K$.
Since $t_{n_k}\to\infty$ and $t'_k\in(t_{n_k},t_{s_k})$, we have $t'_k\to\infty$.
	Moreover,
	\[
	t'_k = k_\Omega(\gamma(0),\gamma(t'_k)) \le \max_{x\in K} k_\Omega(\gamma(0),x) < \infty,
	\]
	a contradiction. Hence $\{\gamma(t_n)\}$ has a unique limit point, and therefore converges in $d$.
\end{proof}

In particular, if $t_n\to\infty$ and $s_n\to\infty$, then $\lim_{n\to\infty}\gamma(t_n)=\lim_{n\to\infty}\gamma(s_n)$.
Therefore, we have the following lemma, which says that in a quasihyperbolically visible space each quasihyperbolic geodesic ray lands at a point of $\partial_{d}\Omega$.

\begin{lem}\label{AV-P7-lem2}
	Let $(\Omega,d)$ be a bounded minimally nice space. Suppose $(\Omega,d)$ is quasihyperbolically visible. If $\gamma$ is a quasihyperbolic geodesic ray in $\Omega$, then $\lim\limits_{t\to \infty}\gamma(t)$ exists in $\partial_d\Omega$.
\end{lem}

\begin{proof}
	Let $t_n\to\infty$. By Lemma \ref{AV-P7-lem1.1}, the limit $\lim_{n\to\infty}\gamma(t_n)$ exists in the metric $d$, and by Lemma \ref{AV-P7-lem1} it belongs to $\partial_d\Omega$.
	Moreover, if $s_n\to\infty$ is another sequence, then the proof of Lemma \ref{AV-P7-lem1.1} yields
	$\lim_{n\to\infty}\gamma(t_n)=\lim_{n\to\infty}\gamma(s_n)$.
	Hence $\lim_{t\to\infty}\gamma(t)$ exists and lies in $\partial_d\Omega$.
\end{proof}

Further, we prove that equivalent quasihyperbolic geodesic rays land at the same point.

\begin{lem}\label{AV-P7-lem3}
	Let $(\Omega,d)$ be quasihyperbolically visible. Fix $x_0\in \Omega$ and let $\gamma$ and $\sigma$ be two quasihyperbolic geodesic rays with $\gamma(0)=x_0=\sigma(0)$ such that $\gamma \sim \sigma$. Then
	\[
	\lim_{t\to \infty}\gamma(t)=\lim_{t\to \infty} \sigma(t).
	\]
\end{lem}

\begin{proof}
	In view of Lemma \ref{AV-P7-lem2}, both limits exist in $\partial_d\Omega$.
	Assume for a contradiction that
	\[
	\lim_{t\to \infty}\gamma(t)=p\neq q=\lim_{t\to \infty} \sigma(t).
	\]
	Choose a sequence $s_n\to \infty$ such that $\gamma(s_n)\to p$ and $\sigma(s_n)\to q$.
	Let $x_n=\gamma(s_n)$ and $y_n=\sigma(s_n)$; then $x_n\to p$ and $y_n\to q$.
	By Observation \ref{observation}, $\lim\limits_{n \to \infty}k_{\Omega}(x_n,y_n)=\infty$, contradicting $\gamma\sim \sigma$.
\end{proof}

The following lemma will be useful to show that the extension map is continuous.

\begin{lem}\label{continuous}
	Let $(\Omega,d)$ be a bounded minimally nice space. Suppose $(\Omega,d)$ is quasihyperbolically visible. Fix $x_0\in \Omega$.
	\begin{itemize}
		\item[(i)] Let $\gamma_n:[0,\infty)\to (\Omega,k_{\Omega})$ be quasihyperbolic geodesic rays with $\gamma_n(0)=x_0$ for all $n$.
		If $(\gamma_n)$ converges uniformly on compact subsets to a geodesic ray $\gamma:[0,\infty)\to (\Omega,k_{\Omega})$, then
		\[
		\lim_{t\to \infty}\gamma(t)=\lim_{n\to \infty}\lim_{t\to\infty}\gamma_n(t).
		\]
		\item[(ii)] Let $\gamma_n:[0,T_n]\to (\Omega,k_{\Omega})$ be quasihyperbolic geodesics with $\gamma_n(0)=x_0$ for all $n$, where $T_n\to\infty$.
		If $(\gamma_n)$ converges uniformly on compact subsets to a geodesic ray $\gamma:[0,\infty)\to (\Omega,k_{\Omega})$, then
		\[
		\lim_{t\to \infty}\gamma(t)=\lim_{n\to \infty}\gamma_n(T_n).
		\]
	\end{itemize}
\end{lem}

\begin{proof}
Since $\overline{\Omega}^d$ is compact, it suffices to prove the claim along an arbitrary convergent subsequence of the boundary values.
	
	\noindent\textbf{Proof of (i).}
By Lemma \ref{AV-P7-lem2}, for each $n$, let $p_n:=\lim_{t\to\infty}\gamma_n(t)\in\partial_d\Omega$ and let $q:=\lim_{t\to\infty}\gamma(t)\in\partial_d\Omega$. Passing to a subsequence, we may assume $p_n\to p\in\partial_d\Omega$. We show that $p=q$. Assume for a contradiction that $p\neq q$. Since $\gamma$ lands at $q$, for every $j\in \mathbb{N}$, there exists $s_j\to \infty$ such that for every $t\ge s_j$, $d(\gamma(t),q)<1/j$. In particular, $d(\gamma(s_j),q)<1/j$. Since $\gamma_n\to\gamma$ uniformly on $[0,s_j]$, there exists $N(j)$ such that
	$d(\gamma_n(s_j),\gamma(s_j))<1/j$ for all $n\ge N(j)$, and hence
	\[
	d(\gamma_n(s_j),q)<2/j\qquad\text{for all }n\ge N(j).
	\]
Now choose an increasing sequence of indices $n_j\ge N(j)$ such that $p_{n_j}\stackrel{d}{\to} p$
and $d(p_{n_j},p)<1/j$ for each $j$. For each $j$, since $\gamma_{n_j}(t)\to p_{n_j}$ as $t\to\infty$, pick $t_j>s_j$ such that
$d(\gamma_{n_j}(t_j),p_{n_j})<1/j$. Then
\[
d(\gamma_{n_j}(t_j),p)\le d(\gamma_{n_j}(t_j),p_{n_j})+d(p_{n_j},p)<2/j.
\]
Set $y_j:=\gamma_{n_j}(s_j)$ and $x_j:=\gamma_{n_j}(t_j)$. Then $y_j\to q$ and $x_j\to p$. By quasihyperbolic visibility, there exists a compact set
$K\subset\Omega$ such that every quasihyperbolic geodesic joining $x_j$ to $y_j$ intersects $K$.
In particular, the geodesic segment $\gamma_{n_j}|_{[s_j,t_j]}$ intersects $K$, so we may choose
$c_j\in[s_j,t_j]$ with $\gamma_{n_j}(c_j)\in K$.
Since $\gamma_{n_j}$ is parametrized by $k_\Omega$-arclength and $\gamma_{n_j}(0)=x_0$,
\[
c_j=k_\Omega(x_0,\gamma_{n_j}(c_j))\le \max_{x\in K}k_\Omega(x_0,x)=:M<\infty.
\]
But $c_j\ge s_j\to\infty$, a contradiction. Hence $p=q$, proving (i).
 
	\noindent\textbf{Proof of (ii).}
	Let $x_n:=\gamma_n(T_n)\in\Omega$. Since $\overline{\Omega}^d$ is compact, passing to a subsequence, we may assume that $x_n\to p\in\overline{\Omega}^d$.
	Since $T_n=k_\Omega(x_0,x_n)\to\infty$ and $k_\Omega$ induces the original topology on $\Omega$, we have $p\in\partial_d\Omega$.
	Let $q:=\lim_{t\to\infty}\gamma(t)\in\partial_d\Omega$. Our aim is to show that $p=q$. If not, choose $s_j\to\infty$ such that $d(\gamma(s_j),q)<1/j$ for each $j$.
		Since $\gamma_n\to\gamma$ uniformly on $[0,s_j]$, for each $j$ there exists $N(j)$ such that
		$d(\gamma_n(s_j),q)<2/j$ for all $n\ge N(j)$.
		Also, since $x_n\to p$, for each $j$ there exists $N'(j)$ such that $d(x_n,p)<1/j$ for all $n\ge N'(j)$. Pick an increasing sequence $n_j\ge \max\{N(j),N'(j)\}$ such that $T_{n_j}>s_j$ (this is possible since $T_n\to \infty$). Set $y_j:=\gamma_{n_j}(s_j)$ and $x_j:=\gamma_{n_j}(T_{n_j})$. Then $y_j\to q$ and $x_j\to p$. By quasihyperbolic visibility, there exists a compact set $K\subset\Omega$ such that every quasihyperbolic geodesic joining $x_j$ to $y_j$ intersects $K$.
	In particular, the segment $\gamma_{n_j}|_{[s_j,T_{n_j}]}$ intersects $K$, so choose
	$c_j\in[s_j,T_{n_j}]$ with $\gamma_{n_j}(c_j)\in K$. As above,
	\[
	c_j=k_\Omega(x_0,\gamma_{n_j}(c_j))\le \max_{x\in K}k_\Omega(x_0,x)=M<\infty,
	\]
	but $c_j\ge s_j\to\infty$, a contradiction. Hence $p=q$ and (ii) follows.
\end{proof}

\begin{proof}[\textbf{Proof of Theorem \ref{Main-1}}]
	The implication (i)$\Rightarrow$(ii) follows from \cite[Theorem 1.1]{Allu-2023}. Conversely, assume that $\Omega$ is quasihyperbolically visible. Fix $z_0\in\Omega$.
		For $\xi\in\partial_G\Omega$, choose a representative quasihyperbolic geodesic ray $\gamma_\xi$ with $\gamma_\xi(0)=z_0$ and set
		\[
		\Phi(\xi):=\lim_{t\to\infty}\gamma_\xi(t)\in\partial_d\Omega,
		\]
		which exists by Lemma \ref{AV-P7-lem2} and is independent of the representative by Lemma \ref{AV-P7-lem3}. For $z\in\Omega$ set $\Phi(z)=z$.
	
	To prove continuity of $\Phi:\overline{\Omega}^G\to\overline{\Omega}^d$, let $x_n\stackrel{Gromov}{\to}x$ in $\overline{\Omega}^G$.
	By definition of the Gromov topology (see Definition \ref{Gromov topology}), we can represent $x_n$ by generalized quasihyperbolic geodesics $\gamma_n$ based at $z_0$ whose endpoints are $x_n$, such that every subsequence admits a further subsequence converging uniformly on compact subsets to a generalized geodesic $\gamma$ ending at $x$.
	If $x\in\Omega$, then $\gamma$ is a finite geodesic and $\Phi(x_n)\to x$ follows immediately.
	If $x\in\partial_G\Omega$, then $\gamma$ is a geodesic ray in the class $x$, and Lemma \ref{continuous}\,(i)--(ii) yields $\Phi(x_n)\to \Phi(x)$.
	Hence $\Phi$ is continuous.
	
	Since $\overline{\Omega}^G$ is compact and $\overline{\Omega}^d$ is Hausdorff, $\Phi(\overline{\Omega}^G)$ is compact and therefore closed in $\overline{\Omega}^d$.
	Moreover, $\Phi(\Omega)=\Omega$, which is dense in $\overline{\Omega}^d$, so $\Phi$ is surjective.
	
	Now, we only need to show that $\Phi$ is a homeomorphism. To show that $\Phi$ is a homeomorphism, it remains to prove that $\Phi$ is injective. Let $\xi,\eta\in\partial_G\Omega$ with $\xi\neq\eta$.
		Choose a geodesic line $\Gamma:(-\infty,\infty)\to(\Omega,k_\Omega)$ with
		$\Gamma(-\infty)=\xi$ and $\Gamma(\infty)=\eta$.
		Set
		\[
		p:=\lim_{t\to\infty}\Gamma(t)\in\partial_d\Omega,
		\qquad
		q:=\lim_{t\to-\infty}\Gamma(t)\in\partial_d\Omega,
		\]
		which exist by applying Lemma \ref{AV-P7-lem2} to the rays $t\mapsto\Gamma(t)$ and $t\mapsto\Gamma(-t)$.
		We claim that $p\neq q$. Assume for a contradiction that $p=q$.
		Choose sequences $t_n\to\infty$ and $s_n\to-\infty$ and set $x_n:=\Gamma(t_n)$ and $y_n:=\Gamma(s_n)$.
		Then $x_n\to p$ and $y_n\to p$ in $d$, hence $d(x_n,y_n)\to0$.
		By quasiconvexity, for each $n$ there exists a curve $\sigma_n$ joining $x_n$ to $y_n$ with
		$\ell_d(\sigma_n)\le A\,d(x_n,y_n)$.
		By the Gehring--Hayman theorem, the quasihyperbolic geodesic segment $\Gamma[x_n,y_n]$ satisfies
		\[
		\ell_d(\Gamma[x_n,y_n])\le C_{gh}\,\ell_d(\sigma_n)\le C_{gh}A\,d(x_n,y_n)\xrightarrow[n\to\infty]{}0.
		\]
		Fix $R>0$. For $n$ large enough we have $s_n\le -R\le R\le t_n$, hence the segment $\Gamma([-R,R])$
		is contained in $\Gamma[x_n,y_n]$, and therefore
		\[
		\ell_d(\Gamma([-R,R]))\le \ell_d(\Gamma[x_n,y_n])\to0.
		\]
		Thus $\ell_d(\Gamma([-R,R]))=0$ for every $R>0$, which forces $\Gamma$ to be constant, a contradiction.
		Hence $p\neq q$, and therefore $\Phi(\xi)\neq\Phi(\eta)$. Thus $\Phi$ is a continuous bijection from the compact space $\overline{\Omega}^G$ to the Hausdorff space $\overline{\Omega}^d$,
	and hence $\Phi$ is a homeomorphism. This completes the proof of Theorem \ref{Main-1}.
\end{proof}

\section{Proof of Theorem \ref{Main-2}}
To prove Theorem \ref{Main-2}, we need the following important remark.
\begin{rem}\label{imp-rem}
	It is often convenient to parametrize a curve $\gamma$ in $(\Omega,d)$ by quasihyperbolic length. We say that $g:[0,r]\to \Omega$ is a quasihyperbolic parametrization of a curve $\gamma:[a,b]\to (\Omega,d)$ if $\ell_k\big(g|_{[0,t]}\big)=t$ for all $t\in [0,r]$. Every rectifiable curve $\gamma$ in $(\Omega,d)$ has a quasihyperbolic parametrization $g:[0,r]\to \Omega$ with image $\gamma([a,b])$, and $g$ satisfies the Lipschitz condition
	$$d(g(s),g(t))\le M|s-t|,$$
	where $M=\max\{\delta_{\Omega}(x):x\in \gamma([a,b])\}$. Since the metric derivative for Lipschitz curves exists a.e. (for details we refer to \cite{Ambrosio-1990}, \cite[p.~2 and p.~24]{Book}, and \cite[Theorem 3.2]{Ambrosio-2000}), we have
	\begin{itemize}
		\item[(i)] $r=\ell_k(g)$,
		\item[(ii)] for a.e.\ $t\in[0,r]$ we have $|g'|(t)=\delta_{\Omega}(g(t))$, and for all $0\le s<t\le r$,
		$$d(g(s),g(t))\le \int_{s}^{t}|g'|(\tau)\,d\tau.$$
	\end{itemize}
	Here $|g'|(t)$ denotes the metric derivative in $(\Omega,d)$. If $\gamma$ is a quasihyperbolic geodesic, then $g:[0,r]\to (\Omega,k_{\Omega})$ is an isometry and we say that $g$ is a quasihyperbolic geodesic from $g(0)$ to $g(r)$.
\end{rem}

\begin{proof}[\textbf{Proof of Theorem \ref{Main-2}}]
	Suppose $(\Omega,d)$ is not quasihyperbolically visible. Then there exist two distinct points $p,q\in \partial_d \Omega$ and sequences $\{x_{n}\}, \{y_{n}\} \subset \Omega$ such that $x_{n}\stackrel{d}{\to} p$ and $y_{n}\stackrel{d}{\to} q$ as $n\to \infty$, and a sequence of quasihyperbolic geodesics $\gamma_{n}:[a_n,b_n]\to \Omega$ joining $x_n$ and $y_n$ such that
	\begin{equation}\label{**}
		\max_{t\in[a_{n},b_{n}]}\delta_{\Omega} (\gamma_{n}(t)) \rightarrow 0 \quad \text{as } n \rightarrow \infty.
	\end{equation}
	By a suitable {\it quasihyperbolic parametrization} $g_n:[a_{n},b_{n}] \rightarrow \Omega$ of $\gamma_{n}$ with $g_{n}(a_{n}) = x_{n}$, $g_{n}(b_{n}) = y_{n}$ and $|b_n-a_n|=\ell_{k}(\gamma_{n})$, we can assume that for all $n$ we have $a_n\le 0\le b_n$ and
	\begin{equation}\label{***}
		\max_{t\in[a_{n},b_{n}]} \delta_{\Omega} (g_{n}(t))=\delta_{\Omega}(g_{n}(0)).
	\end{equation}
	Furthermore, by Remark \ref{imp-rem} each $g_n$ satisfies the Lipschitz condition
	$$d(g_n(s),g_n(t))\le M_n|s-t|,\qquad s,t\in [a_n,b_n],$$
	where $M_n=\max\{\delta_{\Omega}(x):x\in \gamma_n([a_n,b_n])\}$. In view of \eqref{**}, each $g_n$ is $M$-Lipschitz with $M=\sup_{n\ge 1} M_n<\infty$. Therefore, by the Arzel\`a--Ascoli theorem, passing to a subsequence if necessary, we may assume:
	\begin{itemize}
		\item[(i)] $a_{n}\rightarrow a \in [-\infty,0]$ and $b_{n} \rightarrow b \in [0,\infty]$,
		\item[(ii)] $g_{n}$ converges uniformly on every compact subset of $(a,b)$ to a continuous map $g:(a,b)\rightarrow \overline{\Omega}^d$, and $g_{n}(a_{n}) \rightarrow p$, $g_{n}(b_{n}) \rightarrow q$.
	\end{itemize}
	
	We make the following observation.
	\begin{observation}\label{O}
		$a<b$ and $a$ and $b$ cannot both be finite.
		\begin{itemize}
			\item[(i)] $a<b$ follows from the fact that $d(g_n(a_n),g_n(b_n))\le M|a_n-b_n|$ for all $n$.
			\item[(ii)] By Remark \ref{imp-rem}, $g_n:[a_n,b_n]\to (\Omega,k_{\Omega})$ is an isometry. Hence
			$$|b-a|=\lim_{n\to \infty}|b_n-a_n|=\lim_{n\to \infty}k_{\Omega}(g_{n}(b_n),g_{n}(a_n))=\infty
			\quad (\mbox{by Observation \ref{observation}}).$$
			In particular, $a$ and $b$ cannot both be finite.
		\end{itemize}
	\end{observation}
	
	\noindent\textbf{Claim 1.} $g:(a,b) \rightarrow \overline{\Omega}^{d}$ is a constant map: In view of \eqref{**}, we have that $\delta_{\Omega}(g_n(t))$ converges uniformly to $0$ on compact subsets of $(a,b)$. Then for $s<t$ in $(a,b)$ we have
	\begin{equation}
		d(g(s),g(t)) = \lim_{n \rightarrow \infty} d(g_{n}(s),g_{n}(t))\le \lim\limits_{n\to \infty}M_n|s-t|=0.
	\end{equation}
	Thus $g$ is constant. 
	
	Further, using the quasihyperbolic growth in $\Omega$, we will establish a contradiction by proving the following.
	
	\noindent\textbf{Claim 2.} $g :(a,b) \rightarrow \overline{\Omega}^{d}$ is not a constant map: Since each $g_n:[a_n,b_n]\to(\Omega,k_\Omega)$ is an isometry, for $t\in[a_n,b_n]$ we have
	\begin{eqnarray*}
		|t|=k_{\Omega} (g_{n}(0),g_{n}(t)) &\leq& k_{\Omega} (x_{0},g_{n}(0)) + k_{\Omega} (x_{0},g_{n}(t))\\
		&\leq& \phi\left(\frac{\delta_{\Omega}(x_0)}{\delta_{\Omega}(g_n(0))}\right)+\phi\left(\frac{\delta_{\Omega}(x_0)}{\delta_{\Omega}(g_n(t))}\right)\\
		&\leq& 2\phi\left(\frac{\delta_{\Omega}(x_0)}{\delta_{\Omega}(g_n(t))}\right)\,\,\, (\mbox{by }\eqref{***}).
	\end{eqnarray*}
	In view of Observation \ref{O}, we have only the following two cases to consider:
	\begin{itemize}
		\item[Case 1:] Both $a$ and $b$ are infinite.
		\item[Case 2:] Exactly one of $a$ and $b$ is finite.
	\end{itemize}
	
	Let us first consider Case 1 and assume that $a=-\infty$ and $b=\infty$. Since $b_n\to \infty$, $a_n\to -\infty$, and $\phi$ is strictly increasing with $\phi(t)\to\infty$ as $t\to\infty$, there exist $n_0\in\mathbb{N}$ and constants $A,B>0$ such that:
	\begin{itemize}
		\item[(i)] whenever $t\in (B,b_n]$ and $n\ge n_0$, we have
		$$\frac{|t|}{2}\in \mbox{\textbf{range}}(\phi), \ \text{and hence}\  \delta_{\Omega}(g_n(t))\le \frac{\delta_{\Omega}(x_0)}{\phi^{-1}\left(\frac{|t|}{2}\right)}\,.$$
		\item[(ii)] whenever $t\in [a_n,-A)$ and $n\ge n_0$, we have
		$$\frac{|t|}{2}\in \mbox{\textbf{range}}(\phi), \ \text{and hence}\  \delta_{\Omega}(g_n(t))\le \frac{\delta_{\Omega}(x_0)}{\phi^{-1}\left(\frac{|t|}{2}\right)}\,.$$
	\end{itemize}
	
	Choose $a^*\in (-\infty,-A)$ and $b^*\in (B,\infty)$ such that
	\begin{equation}\label{2}
		d(p,q) >\int_{b^*}^{+\infty} \frac{\delta_{\Omega}(x_0)}{\phi^{-1}\left(\frac{t}{2}\right)}\,dt+\int_{-\infty}^{a^*} \frac{\delta_{\Omega}(x_0)}{\phi^{-1}\left(\frac{|t|}{2}\right)}\,dt.
	\end{equation}
	Therefore,
	\begin{eqnarray*}
		d(g(b^*), g(a^*))&=&\lim_{n\to \infty}d(g_{n}(b^*), g_n(a^*))\\
		&\geq& \limsup_{n \rightarrow \infty} \Big[d(g_{n}(b_{n}),g_{n}(a_{n})) - d(g_{n}(b_{n}),g_{n}(b^*)) - d(g_{n}(a_{n}),g_{n}(a^*))\Big]\\
		&\geq& d(p,q) - \limsup_{n\rightarrow \infty}\int_{b^*}^{b_{n}}|g_{n}'|(t)\,dt - \limsup_{n \rightarrow \infty}\int_{a_n}^{a^*} |g_{n}'|(t)\,dt\\
		&\ge&d(p,q)-\limsup_{n\to \infty}\int_{b^*}^{b_n}\delta_{\Omega}(g_n(t))\,dt-\limsup_{n\to \infty}\int_{a_n}^{a^*}\delta_{\Omega}(g_n(t))\,dt\\
		&\ge&d(p,q)-\limsup_{n \rightarrow \infty}\int_{b^*}^{b_n} \frac{\delta_{\Omega}(x_0)}{\phi^{-1}\left(\frac{t}{2}\right)}\,dt-\limsup_{n \rightarrow \infty}\int_{a_n}^{a^*} \frac{\delta_{\Omega}(x_0)}{\phi^{-1}\left(\frac{|t|}{2}\right)}\,dt\\
		&\ge&	d(p,q) -\int_{b^*}^{+\infty} \frac{\delta_{\Omega}(x_0)}{\phi^{-1}\left(\frac{t}{2}\right)}\,dt-\int_{-\infty}^{a^*} \frac{\delta_{\Omega}(x_0)}{\phi^{-1}\left(\frac{|t|}{2}\right)}\,dt>0 \quad (\mbox{by }\eqref{2}).
	\end{eqnarray*}
	Thus, $g$ is non-constant. For Case 2, following the same steps as in Case 1, we conclude that $g$ is non-constant.
	
	Combining both cases, we find that $g$ is not constant, which contradicts Claim 1. This yields the existence of the compact set $K$. Therefore, $(\Omega,d)$ is quasihyperbolically visible. This concludes the proof of Theorem \ref{Main-2}.
\end{proof}
	
	\subsection*{Acknowledgements}
	We express our gratitude to Professor Pekka Koskela for introducing us the paper \cite{Lammi-2011} in relation to Problem \ref{GB-MB} during the research visit of the second named author at the department of mathematics and statistics, Jyv\"askyl\"a university, Finland. The second named author thanks Prime Minister's Research Fellowship (Id: 1200297), Govt. of India for their support. We thank the anonymous referee for a careful reading of the manuscript and for many helpful comments and suggestions that substantially improved the exposition and corrected several inaccuracies.

\end{document}